\documentclass[11pt]{article}
\usepackage{amsmath,amssymb,amsthm,amscd}
\begin{document}
\newtheorem{thm}{Theorem}[section]
\newtheorem{cor}[thm]{Corollary}
\newtheorem{lem}[thm]{Lemma}
\newtheorem{coranddef}[thm]{Corollary and Definition}
\newtheorem*{conj}{Conjecture}
\newtheorem*{thm*}{Theorem}
\theoremstyle{definition}
\newtheorem{defn}[thm]{Definition}
\newtheorem{nota}[thm]{Notation}
\newtheorem*{conv*}{Convention}
\newtheorem{exa}[thm]{Example}
\newtheorem*{subl*}{Sublemma}
\theoremstyle{remark}
\newtheorem{rem}[thm]{Remark}
\newcommand{\ootimes}{\otimes\cdots\otimes}
\newcommand{\ooplus}{\oplus\cdots\oplus}
\newcommand{\lm}{\lambda}
\newcommand{\bbz}{\mathbb{Z}}
\title{Footnotes to a Paper of Domokos, I}
\author{Allan Berele\\ Department of Mathematics\thanks{Partially supported by a Faculty Research Grant from DePaul University}\\ DePaul University\\ Chicago, IL 60614}
\maketitle
\begin{abstract}{We use Regev's double centralizer theorem from~\cite{R} to study invariants of matrices with
group actions.  This theory then has applications to trace identities and Poincar\'e series.}\end{abstract}

\smallskip
Although our constructions will work over any field, the double centralizer theorem
requires characteristic zero and we will take $F$ to be a characteristic zero field
throughout.

In \cite{R} Regev proved the following double centralizer theorem. Let $V$ be a finite
dimensional vector space over a characteristic zero field~$F$, let $A$
be a semisimple subalgebra of $GL(V)$, and let $G$ be the
group of units in $End_A(V)$.  Then
Regev proved that $G$ and the wreath product $A\sim S_k$ have the
double centralizer property in $End(V^{\otimes n})$.  Knowing how
the classical Schur-Weyl duality had been used to study invariant
theory and trace identities, Domokos applied Regev's theory to study
the invariant theory of matrices broken into rectangular blocks.  These
correspond to the case of $V=V_1\ooplus V_t$ and $G=GL(V_1)
\times\cdots\times GL(V_t)$, and Domokos explored the invariant theory
and the trace identities corresponding to~$G$.  We will describe
Domokos's theory in more detail in the course of this paper.

Although one could hardly fault Domokos for choosing the most
natural and important applications of Regev's theory to investigate,
there are a number of applications he didn't consider we believe
worth exploring.  These include
\begin{itemize}
\item Computing the Poincar\'e series of the universal algebras, 
especially using complex integrals; and computing the pure
and mixed trace cocharacters.
\item Generalizing Procesi's embedding theorem from~\cite{P}.
\item Dealing with other $G\subseteq GL(V)$.
\end{itemize}
Before jumping into the body of the paper we give an example of
the last two application, which seems especially novel.  Let
$GL_2(F)$ embed into $M_6(F)$ via $$a\longrightarrow\begin{pmatrix}
a&0&0\\ 0&a&0\\ 0&0&a\end{pmatrix} .$$
Let $A\cong M_3(F)$ be the commuting ring.
We form the free $A$-algebra $A\langle X\rangle$ in which elements 
of~$X$ are assumed to commute with elements of $F$ but not with elements of $A$.
This would be the set of $A$-polynomials and we could ask
about $A$-identities for $M_6(F)$.  However, our interest will
 be in the pure or mixed trace identities.  We enlarge $A\langle
X\rangle$ to an algebra with trace in the natural way to form
$A$-trace polynomials, and get these two
theorems.  
\begin{thm*} Let $\Phi:M_6(F)^n\rightarrow F$ be invariant under
simultaneous conjugation from $GL_2(F)$, considered as a subring of
$ M_6(F)$ as above.  Then
$\Phi$ is equal to a pure trace $M_3(F)$ polynomial.  This means that
the algebra of invariant maps is generated by the maps
$$(x_1,\ldots,x_n)\longrightarrow tr(a_1x_{i_1}\cdots a_kx_{i_k}),$$
where $1\le i_\alpha\le n$ and the $a_i\in A$. Likewise, the
mixed trace $M_3(F)$ polynomials give all invariants functions
from $M_6(F)^n$ to $M_6(F)$.
\end{thm*}
\begin{thm*} Let $f_{i}$, $i=1,2,3$ be image of the matrix unit
$e_{ii}\in M_3(F)\equiv A$ in $ M_6(F)$, so that $f_i M_6(F)f_i$ is isomorphic to 
$M_2(F)$; and let $CH$ be the degree two Cayley-Hamilton
function, so that $CH$ will be an identity for each
 $f_i M_6(F)f_i$.  Then all $A$ trace identities of $M_6(F)$
will be consequences of these three.
\end{thm*}
\begin{thm*} Let $S$ be an algebra with $A$-action and trace which satisfies
the $A$-trace identities of the previous theorem.  Then $S$ can be embedded
into $6\times6$~matrices over a commutative ring, and this embedding will
preserve both $A$-action and trace.
\end{thm*}
In a subsequent paper we hope to explore the $\mathbb{Z}_2$-graded
analogues of the topics in \cite{D} and the current paper.

\section{$E$-Polynomials}
The first case we consider,
following Domokos, will be algebras $R$ with~1,
 in which $1=e_1+\cdots+e_t$, a specified sum of orthogonal 
idempotents.  Let $F\langle X,E\rangle$ be the free algebra
on the set $X\cup E$ where $X$ is a set of indeterminants
and $E$ is the set $\{E_1,\ldots,
E_t\}$ modulo the relations that the $E_i$ are orthogonal
idempotents summing to~1.  Note that $F\langle X,E\rangle$ is
spanned by the monomials $$E_{i_1}X_{j_1}E_{i_2}X_{j_2}\cdots X_{j_s}E_{j_{s+1}}.$$
  This algebra has the universal
property that given any algebra $R$ with b$t$ orthogonal idempotents summing to~1, 
as above, and given any set theoretic map $f:X\rightarrow R$ there
is a unique extension of $f$ to a homomorphism $F\langle X,E\rangle\rightarrow R$ taking each $E_i$ to $e_i$.
Considering the elements of $F\langle X,E\rangle$ to be
$E$-polynomials, we have the natural concepts of $E$-polynomial
identities.  For example, if $R=M_2(F)$, $e_1=e_{11}$
and $e_2=e_{22}$ then the following would be $E$-polynomial
identities:  $$ E_1X_1E_1X_2E_1=E_1X_2E_1X_1E_1,\qquad E_1X_1E_2X_2E_1=E_2X_2E_1X_1E_2.$$
As usual in p.i.~theory the set of identities of an algebra~$R$ is an ideal $Id(R)$ in
$F\langle X,E\rangle$, and the quotient is a universal algebra $U(R)$.  In the
example of $2\times2$ matrices we have been considering, the 
universal algebra can be thought of as either the  algebra generated
by the ordinary  
generic matrice together with $e_{11}$ and $e_{22}$; or as the
algebra generated by all $e_{ii}X_\alpha e_{jj}$:
$$\begin{pmatrix}x_{11}^{(\alpha)}&0\\0&0\end{pmatrix},\ \begin{pmatrix}0&x_{12}^{(\alpha)}
\\0&0\end{pmatrix},\ \begin{pmatrix}0&0\\x_{21}^{(\alpha)}&0\end{pmatrix},\ \begin{pmatrix}
0&0\\0&x_{22}^{(\alpha)}\end{pmatrix}$$
together with $e_{11}$ and $e_{22}$.  We point out here that
our setting is slightly different from Domokos's.  He did not
use the language of $E$-identities and his generic algebra
would have included the $e_{ii}X_\alpha e_{jj}$ but not the
idempotents $e_{ii}$ themselves.

In this paper our main interest will be in trace identities rather than identities. 
In the standard manner we generalize $F\langle X,E\rangle$ to
$F\langle X,E,tr\rangle$ the generic algebra with orthogonal
idempotents and trace.  We note that, because of the
orthogonality of the idempotents, in order for the trace
of a monomial $E_{i_1}x_{j_1}\cdots x_{j_m}E_{i_{m+1}}$
to be non-zero we must have $i_1=i_{m+1}$ and the trace
of the monomial will be
$$tr(E_{i_1}x_{j_1}\cdots E_{i_m}x_{j_m}).$$
It will be convenient to assume that we are given positive
integers~$n_i$ and that the universal algebra satisfies
the identities $tr(E_i)=n_i$ for all~$i$ and so
$tr(1)=\sum n_i$.  Hence, every trace term can be assumed
to have at least one $x_j$ in it.

The ideal of trace identities of~$R$ will be denoted
$Id(R,tr)$, the ideal of pure trace identities
will be denoted $Id(R,ptr)$, and the corresponding
generic algebras as $U(R,tr)$ and $U(R,ptr)$.  In the case of
matrices, $U(M_n(F),ptr)$ will be the subalgebra of the
polynomial ring $F[x_{ij}^{(\alpha)}]_{ij\alpha}$ generated
by the traces of the elements of $U(M_n(F))$,
$$tr(e_{i_1}X_{j_1}\cdots e_{i_m} X_{j_m}).$$
(The $X$ are now capitalized because they are generic
matrices, while the $e$ are lower case to denote specific
idempotents in~$M_n(F)$.)  And, as in the classical case,
$U(M_n(F),tr)$ will be the algebra generated by $U(M_n(F))$
and $U(M_n(F),ptr)$.
\section{Theorems of Regev and Domokos}
Regev's work in \cite{R} generalized the classic Schur-Weyl duality.  Let $V$ be an $n$~dimensional vector space and let $W$ be the tensor power $V^{\otimes k}$.
Then $W$ is a module for the symmetric group via
$$\sigma(v_1\otimes\cdots\otimes v_k)=v_{\sigma^{-1}(1)}\ootimes 
v_{\sigma^{-1}(k)}$$
and the general linear group acts on $W$ by the diagonal action
$$g(v_1\ootimes v_k)=gv_1\ootimes gv_k.$$
It is not hard to see that these two actions commute with each other.  It is less easy
to see, but a theorem with applications in many areas, that each spans the commutator of the other in
$End(W)$.  In Regev's generalization, we start off with a
group $G$ that is a finite direct product of matrix groups
$$G=GL_{n_1}\times\cdots\times GL_{n_t}$$ that acts
on~$V$.  Let
$A$ be the commutator algebra of $G$ in $End(V)$.  The 
group $G$ and the wreath product $A\sim S_k$ act on
$W=V^{\otimes k}$.  Regev proved the double centralizer theorem
for these actions.  Moreover, the algebra $A\sim S_k$ is
semisimple and so decomposes into a sum of simple two-sided
ideals $\oplus I_{\langle\lambda\rangle}$ indexed by
$t$-tuples of partitions, $\langle\lambda\rangle=(\lambda_1,
\ldots,\lambda_t)$; and $I_{\langle\lambda\rangle}$ acts
as zero on~$W$ precisely when some $\lambda_i$ has 
height greater than~$n_i$.

In \cite{D} Domokos applied Regev's theorem to the trace identities of matrices much as Procesi
had applied Schur-Weyl.  He restricted his attention to the case
in which $G=GL(V_1)\oplus\cdots\oplus GL(V_t)$ acts
in the obvious way on $V=V_1\oplus\cdots\oplus V_t$.  In
this case the centralizer algebra $A$ is $t$ copies of the field, $A=F^t$.  
Let $1\in A\subseteq End(V)$ decompose as $e_1+\cdots+e_t$, where 
the $e_i$ are orthogonal idempotents and each
$e_i$ has rank equal to the dimension of~$V_i$.  Here are
Domokos' analogues of the classic theorems of matrix
invariants:
\begin{thm} The pure $E$-trace polynomials give all  functions from $M_n(F)^d$ to the
field~$F$ invariant under simultaneous conjugation from~$G$.\label{thm:d1}
\end{thm}
In the classical case the group $GL_n(F)$ has an action on the 
polynomial algebra $F[x_{ij}^{(\alpha)}]$ defined by
$g(x_{ij}^{(\alpha)})$ equals the $(i,j)$ entry of
$gX_\alpha g^{-1}$.  In the current case we restrict
this action to $G\subseteq GL_n(F)$.
\begin{thm} The pure trace ring $U(M_n(F),ptr)$, which is
the ring generated by the traces of elements of $U(M_n(F))$,
equals the fixed ring of $G$ acting on the polynomial algebra
$F[x_{ij}^{(\alpha)}]$.\label{thm:d2}
\end{thm}
A virtue of the language of $E$-polynomials is that makes it easy to pass
from pure trace identities to mixed ones.  Let $f(x_1,\ldots,x_{k+1})$
be a multilinear pure trace $E$-polynomial.  The $f$ can be written in
the form $$f=\sum_\alpha g_\alpha tr(u_\alpha x_{k+1})=
tr(\sum_\alpha g_\alpha u_\alpha x_{k+1}),$$ where each $g_\alpha$
is a pure trace $E$-polynomial and each $u_\alpha$ is an $E$-polynomial,
and so $\sum g_\alpha u_\alpha$ is a multilinear, mixed trace
$E$-polynomial in $x_1,\ldots,x_k$.  Moreover, since the trace is
non-degenerate we get the following corollaries to Theorems~\ref{thm:d1}
and~\ref{thm:d2}.
\begin{cor} The mixed $E$-trace polynomials give all functions from $M_n(F)^d$
to $M_n(F)$ invariant under simultaneous conjugation from $G$.\label{cor:2.3}\end{cor}
The group $GL_n(F)$, and by restriction the group~$G$, acts on $n\times n$ matrices
over the polynomial ring $F[x_{ij}^{(\alpha)}]$.  We consider
$$M_n(F[x_{ij}^{(\alpha)}])=F[x_{ij}^{(\alpha)}]\otimes M_n(F).$$  The action
of $GL_n(F)$ on the first factor is as above; and the action on the second
factor is via $g(a)=g^{-1}ag$. 
\begin{cor} The (mixed) trace ring $U(M_n(F),tr)$, which is
the ring generated by  $U(M_n(F))$ together with the pure trace ring
$U(M_n(F),ptr)$,
equals the fixed ring of $G$ acting on the matrix algebra
$M_n(F[x_{ij}^{(\alpha)}])$.\label{cor:2.4}\end{cor}
Finally, Domokos also generalizes the Razmyslov-Procesi theorem.  His result is based
on the multilinear Cayley-Hamilton identity, $CH_n(x_1,\ldots,x_n)$, a pure
trace polynomial equivalent to the usual Cayley-Hamilton, see~\cite{GZ}.
\begin{thm}
The ideal of pure $E$-trace identities for $M_n(F)$ is generated by
the Cayley-Hamilton identities restricted to each diagonal component
$$CH_{n_i+1}(e_{ii}X_1e_{ii},\ldots,e_{ii}X_{n_i+1}e_{ii})$$\label{thm:d3}
\end{thm}
Now that we have extended Domokos's results to mixed traces, we can easily extend this
one also.  Define the mixed trace polynomial $MCH_n(x_1,\ldots,x_{n-1})$ via
$$tr(MCH_n(x_1,\ldots,x_{n-1})x_n)=CH_n(x_1,\ldots,x_n).$$  Then by
the non-degeneracy of the trace we have
\begin{thm}
The ideal of mixed $E$-trace identities for $M_n(F)$ is generated by
the mixed Cayley-Hamilton identities restricted to each diagonal component
$$MCH_{n_i}(e_{ii}X_1e_{ii},\ldots,e_{ii}X_{n_i+1}e_{ii})$$\label{thm:d4}
\end{thm}

\section{Cocharacters and Poincar\'e Series}
The free algebras $F\langle X,E\rangle$, $F\langle X,E,tr\rangle$
and $F\langle X,E,ptr\rangle$ each have an $m$-fold grading
with $m=|E|^2\times|X|$ in which
$e_ix_\alpha e_j$ has grading $(0,\ldots,0,1,0\ldots)$, where the~1
is in the position corresponding to $(i,j,\alpha)$.  So, for example,
the degree of $$tr(e_{11}X_3e_{22}X_2e_{44}X_1)=tr(e_{11}X_3e_{22}e_{22}X_2e_{44}e_{44}X_1e_{11})$$
would be 1 in each of the degrees $(1,2,3)$, $(2,4,2)$ and $(4,1,1)$.
  The ideals of identites
of  algebras are homogenous ideals and so the quotients,
which are the universal algebras, are also graded.  This implies
that they will have Poincar\'e series in~$m$ variables which we may call
$t(i,j,\alpha)$.  
For fixed $i,j$ the Poincar\'e series will be symmetric
in the $t(i,j,\alpha)$ and so the Poincar\'e series can be written as
$$\sum_{i,j=1}^{|E|}\sum_{n_{ij}=0}^\infty \sum_{\lambda_{ij}\vdash n_{ij}}m(\lm_{11},\lm_{12}\ldots,\lm_{|E|,|E|})
S_{\lm_{11}}(t(1,1,\alpha)) S_{\lm_{12}}(t(1,2,\alpha)).\cdots$$
Note that the coefficients $m(\lambda)$ depends only on the algebra~$R$ and the
partitions and not on the number of variables, except to the 
extent that the Schur function $S_{\lm_{ij}}$ will be zero unless the 
number of variables $t(i,j,\alpha)$ is greater than or equal
to the height of the partition~$\lm_{ij}$.

We  work first with the Poincar\'e series for
the pure trace ring which we will denote
$P(t,ptr)$ or $P(t,ptr;n_{11},\ldots,n_{ss})$ if we 
need to be more specific.

For example, in the case of  $M_2(F)$ with idempotents
$e_{11}$ and $e_{22}$ the trace ring will be the
ring generated by all $x_{11}^{(\alpha)}$, all $x_{22}^{(\alpha)}$
and all products $x_{12}^{(\alpha)}x_{21}^{(\beta)}$.  If $|X|=1$
the Poincar\'e series would be $$(1-t(1,1,1))^{-1}(1-t(2,2,1))^{-1}
(1-t(1,2,1)t(2,1,1))^{-1}.$$  If $(\lm_{11},\lm_{22},\lm_{12},\lm_{21})=(a,b,c,d)$ is 
a quadruple of one part partitions, then $m(a,b,c,d)$ is zero if $c\ne d$
and it equals~1 if $c=d$.  

This is perhaps a good place to point out there is a parallel 
approach using symmetric group characters.  For fixed
$\sum n_{ij}=n$ one may consider the space $V$ of all non-commutative polynomials
multilinear in $x(i,j,\alpha)$ where for each $(i,j)$, $\alpha$ takes
values from~1 to~$n_{ij}$.  Then~$V$ will be a module for
the direct sum of  symmetric groups $\oplus S_{n_{ij}}$.  As usual, the set of identities for an
algebra~$R$ will be a submodule for $\oplus S_{n_{ij}}$ and hence 
the quotient will
also be a module.  The characters of these modules would constitute the
cocharacter sequence, or to be more precise, multisequence.  Irreducible characters of $\oplus S_{n_{ij}}$
are indexed by tuples of partitions $\lm_{ij}\vdash n_{ij}$.  Again
just as in the classical case, the multiplicity of the irreducible character $\times_{i,j} \chi^{n_{ij}}$ 
in the cocharcter is equal to the multiplicity $m(\lm_{11},\ldots,
\lm_{tt})$ of the corresponding product of Schur functions
in the Poincar\'e series.  Moreover, if $e_iAe_j$ has dimension
less than or equal to some~$d$ then the algebra will satisfy
a Capelli identity of the form
$$\sum (-1)^\sigma (e_ix_{\sigma(1)}e_j)y_1 (e_ix_{\sigma(2)}e_j)y_2\cdots
y_de_d(e_ix_{\sigma(d+1)}e_j).$$
This will mean that $m(\langle \lm\rangle)$ will be zero if $\lm_{ij}$ has
hieght greater than~$d$.  

Returning to the case of $A=M_2(F)$ with two
idempotents, each $e_iM_2(F)e_j$ is one dimensional so the cocharacter
contains only partitions with height less than or equal to~1.  Hence, the
Poincar\'e series in any number of variables is completely determined
by the Poincar\'e series we already calculated in which there is one
variable of each type.  The general Poincar\'e series would be 
$$\sum_{a,b,c=0}^\infty S_{(a)}(t(1,1,{\alpha}))S_{(b)}(t(2,2,\beta))
S_{(c)}(t(1,2,\gamma))S_{(c)}(t(2,1,\delta)).$$

  Thanks to Domokos' invariant theory we are able
to compute the Poincar\'e series in the case of pure
trace identities of matrices.  The main tool is Weyl's character
formula which we now review.

Define the measure $d\nu(z_1,\ldots,z_n)$ to be
$$(n!)^{-1}(2\pi i)^{-n}{\prod_{1\le i\ne j\le n}\left(1-\frac{z_i}{z_j}\right)}\frac{dz_1}{z_1}\wedge\cdots\wedge\frac{dz_n}{z_n}.$$  Let $M$ be a module
for $GL_n(F)$ with character $f(z_1,\ldots,z_n)$, meaning that 
$f(z_1,\ldots,z_n)$ is the trace of a generic diagonal matrix acting on
$M$.  Then Weyl's character formula says that the dimension of the
space of $GL_n(F)$-invariants of~$M$ equals
$$\oint_T f(z_1,\ldots,z_n)d\nu(z_1,\ldots,z_n),$$
where $T$ is the torus $|z_i|=1$, for all~$i$.   A useful special case, more or
less equivalent to the general case is gotten from the case in which $M$ is the tensor product of
an irreducible module with the dual of another irreducible module:\begin{equation}\label{eq:orth}
\oint_T S_\lm(z_1,\ldots,z_n)S_\mu(z_1^{-1},\ldots,z_n^{-1}) d\nu=\delta_{\lm,\mu}.\end{equation}
  For our purposes it will
be better to use the following version of Weyl's character formula.  It
is easy to see that it is a consequence.
\begin{thm}[Weyl's character formula 1] Let $M=\oplus M_\alpha$ be a graded module for
$GL_n(F)$. Let $f_\alpha(z)$ be the character of $M_\alpha$ and
let $f(z,x)=\sum_\alpha f_\alpha(z)x^\alpha$.  Then the Poincar\'e series
of the invariant
space of $GL_n(F)$ equals the integral $\oint f(z,x)d\nu(z)$.
\label{thm:Weyl1}\end{thm}
A further generalization replaces $GL_n(F)$ with $G=GL_{n_1}
\times\cdots\times GL_{n_t}$.  The generic diagonal matrix with
entries $z_1,\ldots,z_n$ belongs to~$G$ and we continue to refer to
the trace of this matrix on a module the character of the module. 
\begin{thm}[Weyl's character formula 2] Let $M=\oplus M_\alpha$ be a graded module for
$G$. Let $f_\alpha(z)$ be the character of $M_\alpha$ and
let $f(z,x)=\sum_\alpha f_\alpha(z)x^\alpha$.  Then the
Poincar\'e series of the invariant
space of $GL_n(F)$ equals the integral $\oint f(z,x)d\nu(z)$, where
$$d\nu(z)=d\nu(z_1,\ldots,z_{i_1})d\nu(z_{i_1+1},\ldots,z_{i_1+i_2})
\cdots.$$
\label{thm:Weyl2}\end{thm}
In dealing with the product of general linear groups, it will be useful to
denote by $K$ the product
\begin{equation}K=\prod_{1\le i\ne j\le n_1}(1-\frac{z_i}{z_j})\prod_{n_1+1\le i\ne j\le n_1+n_2}
(1-\frac{z_i}{z_j})\times\cdots\label{eq:K}\end{equation}
Given $1\le i,j\le n$ define $\gamma(i,j)$ to be the unique
$(a,b)$ such that the matrix unit $e_{ij}$ belongs to
$e_aM_n(F)e_b$
The polynomial ring $F[x_{ij}^{(\alpha)}]$ is a module for the
group $G=Gl_{n_1}\times\cdots\times GL_{n_t}$.  The character
is determined by the trace of a generic diagonal matrix
$$D=(diag(z_1,\ldots,z_{n_1}),diag(z_{n_1+1},\ldots,z_{n_1+n_2}),\ldots).$$
This character, exactly as in the classical case, is $\prod_{ij\alpha}[1-\frac{z_i}{z_j}
t(\gamma(i,j),\alpha]^{-1}$, see~\cite{F}.  Also parallel to the classical case, the
Poincar\'e series of the fixed ring can be computed by Weyl's character
formula using Theorem~\ref{thm:d2}.
\begin{thm} The Poincar\'e series of $U(M_n(F),ptr)$ is
$$\oint_T \prod_{i,j,\alpha}[1-\frac{z_i}{z_j}t(\gamma(i,j),
\alpha)]^{-1}\, d\nu.$$\label{thm:pptr}
\end{thm}
If we wish to compute the Poincar\'e series of
the generic mixed trace ring
$U(M_n(F),tr)$ using Corollary~\ref{cor:2.3}
 we use the character of $G$ acting
on $$M_n(F[x_{ij}^{(\alpha)}])\cong M_n(F)\otimes
F[x_{ij}^{(\alpha)}].$$  The character of $M_n(F)$ under
conjugation by~$G$ is $\sum\frac{z_i}{z_j}$ yielding this
theorem.
\begin{thm} The Poincar\'e series of $U(M_n(F),tr)$ is
$$\oint_T \frac{\sum{z_i}{z_j}}{\prod_{i,j,\alpha}[1-\frac{z_i}{z_j}t(\gamma(i,j),
\alpha)]}\, d\nu.$$\label{thm:pmtr}
\end{thm}
 
In \cite{V} Van Den Bergh showed how to evaluate this type of
integrals using graph theory.  Let $\mathcal{G}$ be the directed graph
on~$n$ vertices~$v_i$ and with an edge $e(i,j,\alpha)$ from $v_i$
to $v_j$ for each term $(1-\frac{z_i}{z_j}
t(\gamma(i,j),\alpha))$ for $i\ne j$ in the denominator of the integrand. 
If $C$ is a simple oriented cycle in $\mathcal{G}$ we write $t^C$ for the
product of the variables corresponding to the edges in~$C$.  Here is Van
Den Bergh's theorem~3.4, based on~\cite{St}:
\begin{thm} The Poincar\'e series $P(t,ptr)$ is a rational funtion
whose denominator can be taken to be the product of all
$(1-t^C)$, where $C$ runs over the simple cycles of the graph.
Moreover, if the variables $t(\gamma(i,j),\alpha)$ are
all distinct, then this gives the least denominator.\label{thm:v1}
\end{thm}

Van Den Bergh has an evaluation of
the integral of the types in
 theorems~\ref{thm:pptr} and~\ref{thm:pmtr} that works
under a fairly mild connectivity hypothesis which we won't bother mentioning, but which holds in
all the cases of interest.  The first step in the process
is to replace the integrals in theorems~\ref{thm:pptr} and~\ref{thm:pmtr} with ones of type
\begin{equation}(2\pi i)^{-n}n!^{-1}\oint_T f \prod_{ij\alpha}(1-\frac{z_i}{z_j}
t(i,j,\alpha))^{-1}\, \frac{dz_1}{z_1}\wedge\cdots\wedge\frac{dz_k}k
\label{eq:2},\end{equation}
where in the general case
 $f$ is a Laurant polynomial in~$z$.
If we can evaluate this integral we can evaluate our original one by
specializing each $t(i,j,\alpha)$ to $t(\gamma(i,j),\alpha)$.  Let
$\mathcal{G}_1$ be the corresponding graph, so $\mathcal{G}$
and $\mathcal{G}_1$ differ only in the labeling of the edges.

  Let $\mathcal{T}$ be a spanning
tree in~$\mathcal{G}_1$ rooted at~$v_1$.  This means that there
is a unique path in $\mathcal{T}$ from each $v_i$ to $v_1$.  For
each edge in $\mathcal{T}$ we have an equation $1-\frac{z_i}{z_j}
t(i,j,\alpha)=0$.  Together these equations let us express
any homogeneous degree zero polynomial $f(z)$ in the $z_i$ as a polynomial
in the $t$ which we write as $f|\mathcal{T}$.  And given an edge $e$ not
in $\mathcal{T}$ we write $C(e,\mathcal{T})$ for the unique  cycle
gotten from $\mathcal{T}$ by adjoining~$e$, oriented in the direction of~$e$.  Let $t^C$ be the product of the
 $t(i,j,\alpha)^{\pm1}$ corresponding to the edges in~$C$, where
the exponent will be $1$ or $-1$ depending on whether the 
orientations of the corresponding edge in $C$ and $\mathcal{G}_1$
agree or not.  Here is Van Den Bergh's theorem:
\begin{thm} The integral \eqref{eq:2} equals the sum
$$\sum \frac{K|\mathcal{T}}{\prod_{e\notin \mathcal{T}}
(1-t^{C(e\mathcal{T})})}$$
where $\mathcal{T}$ runs over all spanning trees rooted at~$v_1$.\label{th:3.6}
\end{thm} 
The integrals from theorems~\cite{thm:pptr} and~\cite{pmtr} can be evaluated using
this theorem and then specializing each $t(i,j,\alpha)$ to $t(\gamma(i,j),\alpha)$.
\section{Two Examples}
\subsection{First Example}
In our first example we take $(n+1)\times(n+1)$ matrices with idempotents
$e_1=e_{11}+\ldots+e_{nn}$ and $e_2=e_{n+1,n+1}$. We will compute
the Poincar\'e series for the pure  trace ring of the algebra
generated by the generic elements
 $X_{12}(\alpha)=
e_1X(\alpha)e_2$ and $X_{21}=e_2X(\alpha)e_1$ only.  For ease of
notation we denote $t(1,2,\alpha)=t_{\alpha}$, $t(2,1,\alpha)=u_\alpha$
and $z_{n+1}=w$.  Then $K=\prod(1-\frac{z_i}{z_j})$ where $1\le i\ne
j\le n$ and the Poincar\'e series is given by
\begin{equation}(2\pi i)^{-n}n!^{-1}\oint \frac K{\prod_i(1-\frac{z_i}w t_\alpha)
(1-\frac w{z_i} u_\alpha)}\frac{dw}w\wedge \frac{dz}{z}.\label{eq:4.1}\end{equation}
By Cauchy's theorem
$$\prod(1-\frac{z_i}w t_\alpha)^{-1}=\sum S_\lambda(z_iw^{-1})s_\lm(t_\alpha)=\sum w^{-|\lm|}S_\lm
(z_i)S_\lm(t_\alpha)$$ and 
$$\prod(1-\frac w{z_i} u_\alpha)^{-1}=\sum S_\mu(wz_i^{-1})S_\mu(u_\alpha)
=\sum w^{|\mu|}S_\mu(z_i^{-1})S_\mu(u_\alpha).$$
If we first integrate the product of these two sums with respect to~$w$ then, since
there is no $w$ in the factor of~$K$ we get only the terms with no $w$, hence
$$(2\pi i)^{-n}n!^{-1}\oint K\sum_{\lm,\mu}S_{\lm}(z_i)S_\lm(t_\alpha)
S_\mu(z_i^{-1})S_\mu(u_\alpha)\; \frac{dz}z.$$
The integral of $KS_\lm(z)S_\mu(z^{-1})$
 equals $\delta_{\lm,\mu}$ by equation~\eqref{eq:orth}.  Hence the integral equals
$$\sum S_\lm(t_\alpha)S_\lambda(u_\alpha).$$
This formula was also obtained by Bahturin and Drensky
in section~3 of~\cite{BD}.
Note that the sum will be over only partitions with height less than or equal to~$n$.
So, if the number of $t_\alpha$ and $u_\alpha$ is each less than or
equal to~$n$ then the Poincar\'e series will be
$\prod_{\alpha,\beta}(1-t_\alpha u_\beta)^{-1}$; and if there are more
variables the Poincar\'e series will be a fraction of the form
$$\frac{N(t,u)}{\prod_{\alpha,\beta}(1-t_\alpha u_\beta)}.$$  Knowledge of the numerator
would be very interesting.

For the mixed trace identities in this case we would multiply the numerator of the integral
in~\eqref{eq:4.1} by $(w+\sum z_i)(w^{-1}+\sum z_i^{-1})$.  We leave the computation
to the interested reader and present only the result.  Given any partition $\lm$ of height less
than or equal to~$n$, let $\lm^+$ be the set of partitions of height less than or equal to~$n$
gotten from~$\lm$ by adding~1 to one of the parts.  Then the Poincar\'e series will equal
$\sum m(\lm,\mu)S_\lm(t_\alpha)S_\mu(u_\alpha)$ where the coefficients are given by
$$m(\lm,\mu)=\begin{cases}1,&\mbox{if }\lm\in\mu^+\mbox{ or if }\mu\in\lm^+\\
|\lm^+\cap\mu^+|,&\mbox{if }|\lm|=|\mu|\\
0,&\mbox{otherwise}
\end{cases}$$
\subsection{Second Example}
For our second example we take the algebra $M_n(F)$ with idempotents $e_i=e_{ii}$
so that each $e_iM_n(F)e_j$ will be one dimensional.  Hence, the Poincar\'e
series will involve only partitions of height less than or equal to~1 and we can determine
it completely by computing the integral in Theorem~\ref{thm:pptr} with only one variable $t(i,j)$ for
each $i,j$, i.e., by considering the Poincar\'e
series of the trace ring of the algebra generated by the set of $t_{ij}e_{ij}$.  Noting that $K=1$ the integral is
\begin{equation}(2\pi i)^{-n}\oint \prod_{ij}(1-\frac{z_i}{z_j}t(i,j))^{-1}\; \frac{dz}z.
\label{eq:3}\end{equation}
This integral has an interesting combinatorial interpretation.
Let $\mathcal{C}_n$ be the complete directed graph on $n$ vertices.  An $\mathbb{N}$-%
flow on $\mathcal{C}_n$ would be an assignment of non-negative integers
to each edge called the flow along that edge such that at each vertex the flow 
on the edges leading into that vertex equals the flow leading out.  Although
we cannot evaluate~\eqref{eq:3} explicitly, there are a number of things we can say.  By theorem~\ref{th:3.6} we know that the
Poincar\'e series will equal
\begin{equation}\sum_{\mathcal{T}}\prod_{e\notin\mathcal{T}}(1-t^{C(e,\mathcal{T})})^{-1},\label{eq:3.1}\end{equation} where $\mathcal{T}$ runs over all 
spanning trees rooted at a given vertex~$v_1$.  Using
standard techniques we can now prove these properties 
of the Poincar\'e series.
\begin{thm}\label{th:4.1} Let $P(t_{11},t_{12},\ldots,t_{nn})$ be the 
Poincar\'e series for the trace ring, $T$, of the ring generated
by the set $\{x_{ij}e_{ij}\}$, $i,j=1,\ldots,n$.
\begin{enumerate}
\item  For each cycle $\sigma=(i_1,i_2,\ldots,i_a)\in S_n$
we write $1-t_\sigma=1-\prod_{b=1}^a t_{i_b,i_{b+1}}$ with the convention that $i_{a+1}=i_1.$  Then
$P(t_{11},t_{12},\ldots,t_{nn})$ is a rational function with least denominator
 the product of all such $1-t_\sigma$, $\sigma\in S_n$.
\item The Poincar\'e series will satisfy a functional equation
$$P(t_{11}^{-1},\ldots,t_{nn}^{-1})=-P(t_{11},\ldots,t_{nn})\prod_{ij}t_{ij}.$$
\item If we specialize each $t_{ij}$ to a new variable~$t$, for each $i\ne j$ then
the resulting function has a pole at $t=1$ of order $n^2-1$.
\end{enumerate}
\end{thm}
\begin{proof} The first part is an immediate consequence
of Theorem~\ref{thm:v1}.

For the second part, It is an interesting exercise which
 we leave to the
reader, to show that in
\eqref{eq:3.1}  for each $\mathcal{T}$ that 
$$\prod_{e\notin \mathcal{T}}(1-t^{-C(e,\mathcal{T})})^{-1}
=-\prod_{ij}t_{ij}\prod_{e\notin \mathcal{T}}(1-t^{C(e,\mathcal{T})})^{-1}.$$

The third part follows from an argument in \cite{F}.
The trace ring $T$ can be written as a tensor product
$F[x_{11},\ldots,x_{nn}]\otimes T'$, where $T'$ is the trace
ring of the algebra generated by the $x_{ij}e_{ij}$, with $i\ne j$, and
so  Poincar\'e series can be written as $\prod_i(1-t_{ii})^{-1}P(T')$.
Let $\langle X\rangle$ be the free abelian group 
generated by the $x_{ij}$ with $i\ne j$, let $\langle Z\rangle$ be the free abelian
group generated by $z_1,\ldots,z_n$, and let
$\langle U\rangle$ be the free abelian group generated by
$u$; and consider the exact
sequence
\[ \begin{CD}1@>>> K@>>> \langle
 X\rangle@>f>>
\langle Z\rangle@>g>> \langle U\rangle@>>>1,\end{CD}\]
where $f(x_{ij})=z_iz_j^{-1}$ and $g(z_i)=u$, for all $i,j$.
It follows that the rank of $K$ is $n^2-n+1$.  But $F(K)$ is the 
quotient field of the trace ring $T'$ is  and so $T$ will have GK-dimension $n^2-1$ and the
theorem follows.\end{proof}
Using the notation $P_n$ to display the dependence on~$n$, here are the first few:
\begin{align*}
P_1&=\frac1{(1-t_{11})}\\
P_2&=\frac1{(1-t_{11})(1-t_{22})(1-t_{12}t_{21})}\\
P_3&=\frac{1-t_{12}t_{13}t_{23}t_{32}t_{31}t_{21}}{\prod_i(1-t_{ii})
(1-t_{12}t_{21})(1-t_{13}t_{31})(1-t_{23}t_{32})(1-t_{12}t_{23}t_{31})(1-t_{13}t_{32}t_{23})}
\end{align*}

For the mixed traces, let Poincar\'e series we would multiply the integrand in~\eqref{eq:3}
by $\sum \frac{z_i}{z_j}$.  Call the Poincar\'e series
be $Q_n$. Then  $Q_n$ is a rational function with  the same denominator as~$P_n$ and, just like
$P_n$, if we specialize each $t_{ij}$ with $i\ne j$ to $t$ in $Q_n$
we get a pole of order $n^2-1$ in the resulting
fraction.  However, $Q_n$ does not satisfy a
functional equation as in Theorem~\ref{th:4.1}~.2.
Then $P_1=Q_1$ and 
$Q_2=$ $$\frac{2+t_{12}+t_{21}}{(1-t_{11})
(1-t_{22})(1-t_{12}t_{21})}.$$  The
numerator of $Q_3$ is

\begin{multline*}-3\,t_{1,2}\,t_{1,3}\,t_{2,1}\,t_{2,3}\,t_{3,1}\,t_{3,2}-t_{1,3}\,t
 _{2,1}\,t_{2,3}\,t_{3,1}\,t_{3,2}-t_{1,2}\,t_{2,1}\,t_{2,3}\,t_{3,1}
 \,t_{3,2}\\-t_{1,2}\,t_{1,3}\,t_{2,3}\,t_{3,1}\,t_{3,2}-t_{1,2}\,t_{2,
 3}\,t_{3,1}\,t_{3,2}-t_{1,2}\,t_{1,3}\,t_{2,1}\,t_{3,1}\,t_{3,2}\\-t_{
 1,3}\,t_{2,1}\,t_{3,1}\,t_{3,2}-t_{1,2}\,t_{1,3}\,t_{2,1}\,t_{2,3}\,
 t_{3,2}-t_{1,3}\,t_{2,1}\,t_{2,3}\,t_{3,2}\\-t_{1,2}\,t_{1,3}\,t_{2,1}
 \,t_{3,2}+t_{2,1}\,t_{3,2}+t_{1,3}\,t_{3,2}+t_{3,2}-t_{1,2}\,t_{1,3}
 \,t_{2,1}\,t_{2,3}\,t_{3,1}\\-t_{1,2}\,t_{2,1}\,t_{2,3}\,t_{3,1}-t_{1,
 2}\,t_{1,3}\,t_{2,3}\,t_{3,1}+t_{2,3}\,t_{3,1}+t_{1,2}\,t_{3,1}+t_{3
 ,1}\\+t_{1,2}\,t_{2,3}+t_{2,3}+t_{1,3}\,t_{2,1}+t_{2,1}+t_{1,3}+t_{1,2
 }+3\end{multline*}
\section{$A$-identities}
Regev's theorem concerns a double centralizer between the wreath
product $A\sim S_n$ where $A\subseteq End(V)$ is a semisimple
algebra and the group $\mathcal{G}$ of units in $End_A(V)$.
Domokos limited his attention to the case in which the action is
multiplicity free and $A=F^t$.  This leads to the very natural theory
of $E$-trace identities.  It is possible to reproduce many of Domokos' results
in a more general setting.  Let $F\langle X,A\rangle$ be the amalgomated
product $A*F\langle X\rangle$.  Elements of $F\langle X,A\rangle$ are 
linear combinations of terms of the form $a_1x_{i_1}a_2\cdots a_mx_{i_m}a_{m+1}$.
We will call these $A$-polynomials.  It makes sense to talk about $A$-identities for
any $A$-algebra, in particular for $M_n(F)=End(V)$.
Since the $a_i$ are in $M_n(F)$ it makes sense to substitute arbitrary matrices
for the $x_i$ and to talk about $A$-identities for $M_n(F)$.  Likewise we can 
define $A$-trace polynomials, either mixed or pure.  
\begin{defn}Given $\sigma\in S_k$ whose inverse has cycle decomposition
$$\sigma^{-1}=(i_1,\ldots,i_a)(j_1,\ldots,j_b)\cdots$$ we define the trace monomial
$tr_\sigma$ via
$$tr_\sigma(x_1,\ldots,x_k)=tr(x_{i_1}\cdots x_{i_a})tr(x_{j_1}\cdots x_{j_b})\cdots.$$
   Given $w=(a_1\otimes\cdots\otimes a_k,\sigma)$ a monomial in $A\sim S_k$ we define
$tr_w(x_1,\ldots,x_k)$ to be $tr_\sigma(x_1a_1,\ldots,x_ka_k)$.
More generally, if $w=\sum w_i\in W_k$ is a sum of monomials we let
$tr_w$ be $\sum tr_{w_i}$.\label{def:mult}\end{defn}

Pure $A$-trace polynomials are important
for invariant theory:
\begin{thm} The $A$-pure trace polynomials give all function $f:M_n(F)^k\rightarrow F$ 
invariant under simultaneous conjugation from~$\mathcal{G}$.\label{thm:atr}
\end{thm}
\begin{proof}
By multilinearization we need only consider the multilinear maps.  Identifying
$M_n(F)$ with $End(V)$ for an $n$-dimensional vector space~$V$, we consider
$[(End(V)^{\otimes k})^*]^\mathcal{G}$.  
Now, as $G$-modules
\begin{align*}
(End(V)^{\otimes k})^*&\cong \left((V\otimes V^*)^{\otimes k} \right)^*\\
&\cong (V^{\otimes k}\otimes V^{*\otimes k})^*\\
&\cong End(V^{\otimes k}).
\end{align*}
The last isomorphism requires comment.  Let $T\in End(V^{\otimes k})$.
We associate to~$T$ the functional on $V^{\otimes k}\otimes V^{*\otimes k}$
which takes $\underline{v}\otimes \underline{\phi}$ to $\underline{\phi}
(T(\underline{v}))$.  Hence, the $\mathcal{G}$ invariants come from
$End_{\mathcal{G}}(V^{\otimes k})$ which equals $A\sim S_k$.

To complete the proof let $x_i=v_i\otimes \phi_i\in V\otimes V^*\equiv End(V)$, let
$(a_1\ootimes a_k)\sigma\in A\sim S_k$ and compute
\begin{align*}
\underline{\phi}((a_1\ootimes a_k)\sigma(\underline{v}))&=
\underline{\phi}(a_1v_{\sigma^{-1}(1)}\ootimes a_kv_{\sigma^{-1}(k)})\\
&=\phi_1((a_1v_{\sigma^{-1}(1)})\cdots \phi_k(a_kv_{\sigma^{-1}(k)})\\
&=\phi_1\circ a_1(v_{\sigma^{-1}(1)})\cdots \phi_k\circ a_k(v_{\sigma^{-1}(k)})\\
&=tr_{\sigma}(v_1\otimes\phi_1 a_1,\ldots,v_k\otimes\phi_ka_k)\\
&=tr_{\sigma}(x_1a_1,\ldots,x_ka_k)\\
&=tr_{(a_1\ootimes a_k)\sigma}(x_1,\ldots,x_k)
\end{align*}
and this completes the proof.
\end{proof}

As a corollary to the proof we have
\begin{cor} A multilinear pure $A$-polynomial $tr_w$ is an $A$-identity for
$E=End(V)$ if and only if $w$ is in the kernal of the map $A\sim S_k\rightarrow
End(V^{\otimes k})$.\label{cor:atr}\end{cor}
Theorem \ref{thm:atr} can be restated in terms of generic matrices.  Let $R(A,n,k)$
be the ring generated by the generic $n\times n$ matrices $X_1,\ldots,X_k$,
$X_\alpha=(x_{ij}^{(\alpha)})$, 
together with $A\subseteq M_n(F)$; and let $\bar{C}(A,n,k)$ be the ring
generated by the traces of elements of $R(A,n,k)$.  The group $\mathcal{G}$
acts on the polynomial ring $F[x_{ij}^{(\alpha)}]$ via by defining $g(x_{ij}^{(\alpha)})$
to be the $(i,j)$ entry of $gX_\alpha g^{-1}$.  Just as in the classical case
Theorem~\ref{thm:atr} is equivalent to the following.
\begin{thm} The fixed ring $F[x_{ij}^{(\alpha)}]^{\mathcal{G}}$ equals the
generic trace ring $\bar{C}(A,n,k)$.\end{thm}

Since $A$ is a semisimple algebra it can be written as a direct sum of simple
algebras $A=\sum_{i=1}^t  A_i$.  Let $e_i$ be a two sided idempotent in
$A_i$ and assume $A_i\cong M_{n_i}(F)$.  The pure trace idenities of $M_n(F)$
are generated by the Cayley-Hamilton polynomial
$$C_{n+1}(x_1,\ldots,x_{n+1})=\sum_{\sigma\in S_{n+1}}(-1)^\sigma tr_\sigma(x_1,
\ldots,x_{n+1}).$$
In our case $A$ will satisfy the identities
\begin{equation}C_{n_i+1}(x_1e_i,\ldots,x_{n_i+1}e_i).\label{eq:ch}\end{equation}
Similar to both the classical case and Domokos' generalization we have the following
theorem
\begin{thm}\label{thm:rp} All pure $A$-trace identities for $M_n(F)$ are consequences of
those in \eqref{eq:ch}, where $i=1,\ldots,t$.\end{thm} 
Since it is an oft told tale at this point, we merely sketch the proof.
By a multilinearization argument, it suffices to prove the theorem for multilinear
identities.    As in the classical
case we have the following lemma.
\begin{lem} Given $w\in A\sim S_k$, $\sigma\in S_k$ and invertible
$b=B_1\otimes\cdots\otimes B_k\in A^{\otimes k}$,
$$tr_{\sigma w\sigma^{-1}}(x_1,\ldots,x_k)=tr_w(x_{\sigma(1)},\ldots,x_{\sigma(k)})$$
and
$$tr_{bwb^{-1}}(x_1,\ldots,x_k)=tr_w(B_1^{-1}x_1B_1,\ldots,B_k^{-1}x_kB_k).$$
\end{lem}
Given an algebra $R$ with $A$-action and trace, let $W_k$ be the set of
$w\in A\sim S_k$ such that $tr_w$ is an identity for~$R$.  The lemma 
shows that $W_k$ is closed under conjugation, but it may not be closed
under right or left multiplication.  We use it to prove the 
following.
\begin{lem} Assume that $W_k$ is a two-sided ideal of~$A\sim S_k$.  Then for all
$m\ge k$ the elements of $(A\sim S_m)W_k(A\sim S_m)$ are consequences
of $W_k$ and hence identities for~$R$.\label{lem:ind}
\end{lem}
\begin{proof}[Sketch of Proof]  Using induction it suffices to do the case of $m=k+1$.
Using the previous lemma it suffices to consider $(A\sim S_m)W_k$.  And, since
$W_k$ is closed under multiplication from $A\sim S_k$, we reduce to 
the product $(a_1\ootimes a_{k+1}\sigma)u$ where $u\in W_k$ and either
$\sigma$ is the identity or a transposition of the form $(i,k+1)$.  In the
 case of $\sigma$ equal to the identity it is not hard to prove that
 $$tr_{a_1\ootimes a_{k+1}u}(x_1,\ldots,x_{k+1})=tr_{a_1\ootimes a_k u}(x_1,
 \ldots,a_k)tr(x_{k+1}a_{k+1}).$$
 Since $W_k$ is closed under multiplication by $A^{\otimes k}$, this is certainly
 a consequence of an element of $W_k$.
 
 Finally, in the case of $\sigma=(i,k+1)$ we re-write $(a_1\ootimes a_{k+1})\sigma$
 as $\sigma(b_1\ootimes b_{k+1})$ and, again use fact that $W_k$ is closed under
 multiplication from $A^{\otimes k}$ to reduce to $(i,k+1)(1\ootimes 1\otimes b_{k+1})u$.
 We leave it to the reader to prove in this case that
 $$tr_{(i,k+1)(1\ootimes 1\otimes b_{k+1})u}(x_1,\ldots,x_{k+1})=
 tr_u(x_1,\ldots,x_k)\vert_{x_i=x_{k+1}b_{k+1}x_i}.$$
 \end{proof}
 
 By Corollary~\ref{cor:atr} $W_k$, the identities of $M_n(F)$ in $A\sim S_k$
 will be a two-sided ideal, identified with the kernal of the map $A\sim S_k\rightarrow
 End(V)$.  The two sided ideals $I_{\langle\lm\rangle}$  $A\sim S_k$ are parameterized by $t$-tuples
 of partitions $\langle \lambda\rangle=(\lm_1,\ldots,\lm_t)$ with $\sum |\lm_i|=k$.
The proof of the following is the same as the proof of  Theorem~2.1(ii) in \cite{D}.
 \begin{thm} The kernal of the map $A\sim S_k\rightarrow End(V^{\otimes k})$
 equals the sum of all ideals $I_{\langle\lm\rangle}$ with some $\lm_i$
 of height greater that $n_i$.\label{thm:ker}\end{thm}
 In light of Lemma~\ref{lem:ind} we need to see how ideals of $A\sim S_k$ induce
 up to $A\sim S_m$.  As Domokos noted, the next lemma follows from the Appendix
 of~\cite{GR}.
 \begin{lem}[The Branching Lemma] The ideal $I_{\langle\lm\rangle}\trianglelefteq (A\sim S_k)$ induces
 up to $A\sim S_m$, $m>k$ as $\oplus I_{\langle\mu\rangle}$ summed over all
 $\langle\lm\rangle\subseteq\langle\mu\rangle$, i.e., $\lm_i\subseteq\mu_i$ for each~$i$.
 \label{lem:branch}\end{lem}
 We now have all the ingrediants we need to prove Theorem~\ref{thm:rp}.
 \begin{proof}[Proof of Theorem~\ref{thm:rp}]
 Identifying $A\sim S_m$ with multilinear pure $A$-trace polynomials as in Defintion~\ref{def:mult}, 
 Corollary~\ref{cor:atr} says that the identities of $M_n(F)$ equals the kernal of the
 map $A\sim S_m\rightarrow End(V^{\otimes m})$.  Next, Theorem~\ref{thm:ker}
 identifies that kernal as the sum of the ideals $I_{\langle\mu\rangle}$ with some
 $\mu_i$ of height greater than~$n_i+1$.
 A partition $\mu_i$ will have height greater than or equal to $n_i+1$ if and only if
 it contains the partition $(1^{n_i+1})$.  By Lemma~\ref{lem:branch} the ideals 
 $I_{\langle\mu\rangle}$ are induced from the ideals $I_{\lm(i)}$
 where $\lm(i)=(\lm_1,\ldots,\lm_i)$ and 
 $$\lm_j=\begin{cases}\mbox{the empty partition}&\mbox{if }j\ne i\\
 (1^{n_i+1})&\mbox{if }j=i.\end{cases}$$
 By Lemma~\ref{lem:ind} all identities of $M_n(F)$ are consequnces of those
 in $I_{\lm(i)}$.  These ideals are one dimensional and correspond to the
 polynomials in~\eqref{eq:ch}.\end{proof}
The results of this section all have analogues for mixed trace
identities and are easily derived from them using the facts that
$f(x_1,\ldots,x_k)$ is a mixed trace identity if and only if
$tr(f(x_1,\ldots,x_k)x_{k+1})$ is a pure trace identity, and every
pure trace identity $g(x_1,\ldots,x_{k+1})$ can be written in the
form $tr(f(x_1,\ldots,x_k)x_{k+1})$ for a pure trace identity $f$.
Since the techniques are standard we merely state the results.
\begin{thm} The $A$-mixed trace polynomials give all function $f:M_n(F)^k\rightarrow M_n(F)$ 
invariant under simultaneous conjugation from~$\mathcal{G}$.
\end{thm}
Let $\bar{R}(A,n,k)$ be the algebra generated by $R(A,n,k)$ and
$\bar{C}(A,n,k)$, which means that it is the algebra with trace 
generated by the generic matrices $X_1,\ldots,X_k$ and by $A$.
\begin{thm} The fixed ring $M_n(F[x_{ij}^{(\alpha)}])^{\mathcal{G}}$
 equals the ring $\bar{R}(A,n,k)$.\label{thm:5.11}\end{thm}
Define the mixed trace Cayley-Hamilton polynomials via
$$ tr(MC_n(x_1,\ldots,x_n)x_{n+1})=C_{n+1}(x_1,\ldots,x_{n+1})$$
\begin{thm} All mixed $A$-trace identities for $M_n(F)$ are 
consequences of the identities $$MC_{n_i}(e_ix_ie_i,\ldots,
i_ix_{n_i}e_i).$$
\end{thm}

 \section{Poincar\'e Series Again}
We consider the case of $G=GL_a(F)=GL(W)$ acting on $V=
W^b$ which we write as $W_1\oplus\cdots\oplus W_b$.  Let $n=ab$. 
  The commuting ring $End_G(V)$ is isomorphic to $M_b(F)$
 with the matrix unit $e_{ij}$ an isomorphism $W_i\rightarrow
W_j$.  For example, if $a=b=2$ then
$$e_{11}\leftrightarrow\begin{pmatrix}1&0&0&0\\0&1&0&0\\
0&0&0&0\\0&0&0&0\end{pmatrix},\ e_{12}\leftrightarrow\begin{pmatrix}0&0&1&0\\0&0&0&1\\
0&0&0&0\\0&0&0&0\end{pmatrix},\mbox{ etc.}$$
Write the matrix unit in $M_{ab}(F)$ corresponding to
$e_{ij}$ in $M_a(F)$ as $E_{ij}$.
As in the classical case, the universal algebra for algebras
satisfying the $G$-identities of $M_{ab}(F)$ is the algebra
generated by all $E_i X_\alpha E_j$ where the $X_\alpha$
are generic $n\times n$ matrices.  Let $U(a,b,k)$ be this universal algebra on~$k$ generators; let $U(a,b,k,ptr)$
be the algebra generated by the traces; and let $U(a,b,k,tr)$ be
the algebra generated by both. Theorems~\ref{thm:atr} and~\ref{thm:5.1} imply
the following.
\begin{thm} $U(a,b,k,ptr)$ is the fixed ring of $G$ acting
on the polynomial ring $F[x_{ij}^{(\alpha)}]$ and 
$U(a,b,k,tr)$ is the fixed ring of~$G$ acting on the
matrix ring over the polynomial algebra $M_n(F[x_{ij}^{(\alpha)}])$.\end{thm}
Using Weyl's integration formula we can express the 
Poincar\'e series for the pure and mixed trace rings as
complex integrals.  Let $D\in GL_a(F)$ be the generic diagonal
matrix $diag(z_1,\ldots,z_a)$.  Then $D$ acts on $W$ as
$$D^{\oplus b}=diag(\underbrace{z_1,\ldots,z_1}_b,\ldots,
\underbrace{z_a,\ldots,z_a}_b).$$   Given $1\le i,j\le n$ we
let $\gamma(i,j)$ be the unique $(s,t)$ such that
$e_{ij}\in E_s M_n(F) E_t$.  We also let $s=\gamma_1(i,j)$
and $t=\gamma_2(i,j)$.  Hence, under the conjugation action
$e_{ij}$ would be an eigenvector for $D$ with eigenvalue
$z_s z_t^{-1}$.  If we grade $F[x_{ij}^{(\alpha)}]$ by $\mathbb{N}^{ka^2}$ as in section~3, then $U(a,b,k,ptr)$
has Poincar\'e series
\begin{equation}(2\pi i)^{-a}a!^{-1}\oint_T \frac{\prod_{i\ne j}(1-\frac{z_i}{z_j})}{\prod_{ij\alpha}(1-\frac{\gamma_1(i,j)}{\gamma_2(i,j)}
t(i,j,\alpha))^{b}}\, \frac{dz_1}{z_1}\wedge\cdots\wedge\frac{dz_a}z_a
\label{eq:3a}\end{equation}
Let $P(a,b,k)$ denote this integral.  In the special case of 
$b=1$ we are reduced to the classical case of $GL_a(F)$ acting on
$a\times a$ matrices and the Poincar\'e series is given by the integral
in theorem~\ref{thm:pptr}.  Comparing the integrands
of that theorem to~\eqref{eq:3a} we have the following.
\begin{thm} Let $P(a,1,bk)$ be a function of the variables
$t(i,j,(m,\alpha))$, where $1\le i,j\le a$, $1\le m\le b$ and
$1\le\alpha\le k$; and let $P(a,b,k)$ be a function of the
variables $t(i,j,\alpha)$, where $1\le i,j\le a$ and $1\le
\alpha\le k$.  Then $P(a,b,k)$ can be computed from
$P(a,1,bk)$ by taking the limit of $t(i,j,(m,\alpha))\rightarrow
t(i,j,\alpha)$ for each variable.\label{thm:6.2}\end{thm}
We can also restate Theorem~\ref{thm:6.2} in terms of cocharacters.  We define an operation on Schur functions via
$$S_\lambda^{\boxtimes b}(t_1,\ldots,t_a)=
S_\lambda(\underbrace{t_1,\ldots,t_1}_b,\ldots,
\underbrace{t_a,\ldots,t_a}_b).$$
This is equivalent to
$$S_\lambda^{\boxtimes b}(t_1,\ldots,t_a)=\sum_{\mu_1\subseteq\mu_2\subseteq\cdots\subseteq\mu_{b-1}
\subseteq\lm}S_{\mu_1}(t_a,\ldots,t_a)S_{\mu_2/\mu_1}(t_a,\ldots,t_a)\cdots S_{\lm/\mu_{b-1}}(t_a,\ldots,t_a).$$
We may define a corresponding operation on $S_n$-characters via
$$(x^\lm)^{\boxtimes b}=\sum_{\mu_1\subseteq\mu_2\subseteq\cdots\subseteq\mu_{b-1}
\subseteq\lm}\chi^{\mu_1}\otimes\chi^{\mu_2/\mu_1}\otimes\cdots \otimes\chi^{\lm/\mu_{b-1}},$$
where the tensor product is the outer tensor product.
\begin{thm} Let the ring of generic $a\times a$ matrices have Poincar\'e series
$\sum m_\lambda S_\lambda(t_1,\ldots,t_k)$.  Then $U(a,b,k)$ has
Poincar\'e series $\sum m_\lambda S_\lambda^{\boxtimes b}(t_1,\ldots,t_k)$.
And if $M_a(F)$ has $m$-th cocharacter $\sum m_\lm \chi^\lm$,
then $M_{ab}(F)$ will have $m$-th $G$-cocharacter $\sum m_\lm (\chi^\lm)^{\boxtimes b}$
\end{thm}
\section{Procesi's Embedding Theorem}
In~\cite{P} Procesi proved the following theorem:
\begin{thm*}[Procesi] Let $R$ be an algebra with trace in
characteristic zero, satisfying all of the trace identities of $M_n(F)$.  Then $R$ has a trace preserving embedding into
$M_n(C)$  for some commutative algebra~$C$.\end{thm*}
As Procesi carefully points out, his proof follows from the
following:
\begin{itemize}
\item The free ring with trace $F\langle X,tr\rangle$ modulo
the trace identities of $M_n(F)$ is isomorphic to the generic
matrix algebra with trace.
\item This free ring is the fixed ring of $G=GL_n(F)$ acting on $M_n(F[x_{ij}^{(\alpha)}
])$.
\item The group $G=GL_n(F)$ is linearly reductive.
\end{itemize}
Using Corollary \ref{cor:2.4} and Theorem~\ref{thm:d4} This immediately gives the following:
\begin{thm} Let $G$ be a linearly reductive subgroup of $GL_n(F)$ and let
$A=End_G(F^n)\subseteq M_n(F)$.  If $R$ is an $A$-algebra with trace satisfying all of the
$A$-trace identities of $M_n(F)$, then $R$ has an embedding into $M_n(C)$ for some
commutative algebra~$C$ which preserves both trace and $A$-action.
\end{thm}
For example, take $G=GL_{n_1}\times\cdots\times GL_{n_t}\subseteq GL_n$,
where $n=\sum n_i$.  Then $A=F^t$ and an $A$-algebra is what we called an $E$-algebra.
\begin{cor} Let $R$ be an $E$-algebra satisfying the multilinear Cayley-Hamilton
idenitites $$CH_{n_i+1}(E_{i}X_1E_{i},\ldots,E_{i}X_{n_i+1}E_{i})$$
as in Theorem~\ref{thm:d3}.  Then $R$ has an $E$-trace preserving embedding
into $M_n(C)$ for some commutative algebra~$C$.\end{cor}

\end{document}